\documentclass[12pt]{amsart}
  \usepackage{latexsym} 
  \usepackage[all]{xy}
  \usepackage{amsfonts} 
  \usepackage{amsthm} 
  \usepackage{amsmath} 
  \usepackage{amssymb}
  \usepackage{pifont}  
  \newcommand{\cC}{{C}}  

  \def\sw#1{{\sb{(#1)}}}

  \def\<{{\langle}} 
  \def\>{{\rangle}}

  \def\eps{\varepsilon}

  \def\note#1{{}}

  \def\note#1{}

  \def\cC{{\mathcal C}}

  \def\rhom#1#2#3{{{\rm Hom}\sb{#1}(#2,#3)}}
  \def\hom#1#2{{{\rm Hom}(#1,#2)}}

  \def\beq{\begin{equation}} 
  \def\eeq{\end{equation}} 
  \def\DC{{\Delta_C}} 
  \def \eC{{\eps_C}} 
  \def\DH{{\Delta_H}} 
  \def \eH{{\eps_H}}

  \def\id{\mathrm{id}}

  \def\ot{{\otimes}}

  \def\roN{\varrho^{N}} 
  \def\Nro{{}^{N}\!\varrho}

  \newcounter{zlist} 
  \newenvironment{zlist}{\begin{list}{(\arabic{zlist})}{ 
  \usecounter{zlist}\leftmargin2.5em\labelwidth2em\labelsep0.5em 
  \topsep0.6ex%\itemsep0.3ex plus0.2ex minus0.3ex 
  \parsep0.3ex plus0.2ex minus0.1ex}}{\end{list}}

  \newcounter{blist}

  \newcounter{rlist}

\def\stac#1{\raise-.2cm\hbox{$\stackrel{\displaystyle\otimes}{\scriptscriptstyle{#1}}$}}

\def\cten#1{\raise-.2cm\hbox{$\stackrel{\displaystyle\widehat{\otimes}}
{\scriptscriptstyle{#1}}$}}

  \def\Label#1{\label{#1}\ifmmode\llap{[#1] }\else 
  \marginpar{\smash{\hbox{\tiny [#1]}}}\fi} 
  \def\Label{\label}

  \newtheorem{lemma}{Lemma} 
   
  \newtheorem{theorem}{Theorem}

  \theoremstyle{definition}

  \theoremstyle{remark}

  \theoremstyle{definition} 
  \swapnumbers
  \newtheorem{statement}{}

  \newcounter{c} 
   
  \newcommand{\etyk}[1]{\vspace{-7.4mm}$$\begin{equation}\Label{#1} 
  \addtocounter{c}{1}} 
  \renewcommand{\]}{\ifnum \value{c}=1 $$\else \end{equation}\fi} 
\begin{document} 

 \title{Hopf-cyclic homology with contramodule coefficients} 
 \author{Tomasz Brzezi\'nski}
 \address{ Department of Mathematics, Swansea University, 
  Singleton Park, \newline\indent  Swansea SA2 8PP, U.K.} 
  \email{T.Brzezinski@swansea.ac.uk}   
    \date{\today} 
  \begin{abstract} 
A new class of coefficients for the Hopf-cyclic homology of module algebras and coalgebras  is introduced. These coefficients, termed {\em stable anti-Yetter-Drinfeld contramodules}, are both modules and {\em contramodules} of a Hopf algebra that satisfy certain compatibility conditions. 
  \end{abstract} 
  \maketitle

 \begin{statement} {\bf Introduction.}  
 It has been demonstrated in  \cite{HajKha:Hop},  \cite{JarSte:coh} that the Hopf-cyclic homology developed by Connes and Moscovici \cite{ConMos:cyc} admits a class of non-trivial coefficients. These coefficients, termed {\em anti-Yetter-Drinfeld  modules} are modules and comodules of a Hopf algebra satisfying a compatibility condition reminiscent of that for cross modules. The aim of this note is to show that the Hopf-cyclic (co)homology of module coalgebras and module algebras also admits coeffcients that are modules and {\em contramodules} of a Hopf algebra with a compatibility condition.
 
 All (associative and  unital) algebras, (coassociative and counital) coalgebras in this note are over a field $k$. The coproduct in a coalgebra $C$ is denoted by $\Delta_C$, and counit by $\eps_C$. A Hopf algebra $H$ is assumed to have a bijective antipode $S$. We use the standard Sweedler notation for coproduct $\Delta_C(c) = c\sw 1\ot c\sw 2$, $\Delta^2_C (c) = c\sw 1\ot c\sw 2\ot c\sw 3$, etc., and for 
 the left coaction $\Nro$ a $C$-comodule $N$, $\Nro(x) = x\sw{-1} \ot x \sw 0$ (in all cases summation is implicit). $\hom V W$ denotes the space of $k$-linear maps between vector spaces $V$ and $W$. 
 \end{statement}

\begin{statement} {\bf Contramodules.}   The notion of a {\em contramodule} for a coalgebra was introduced in \cite{EilMoo:fou}, and discussed in parallel with that of a comodule. A {\em right contramodule} of a coalgebra $C$ is a vector space $M$ together with a $k$-linear map $\alpha : \hom C M\to M$ rendering the following diagrams commutative
$$
\xymatrix{
\hom  C {\hom  C M} \ar[rrrr]^-{\hom  C {\alpha}}\ar[d]_\Theta &&&& \hom C M \ar[d]^{\alpha} \\
\hom {C\ot C} M \ar[rr]^-{\hom  \DC M} && \hom  C M \ar[rr]^-{\alpha} && M ,}
$$
$$
\xymatrix{\hom  k M \ar[rr]^-{\hom  \eC M} \ar[rd]_\simeq && \hom  C M\ar[dl]^{\alpha} \\
& M , & }
$$
where $\Theta$ is the standard isomorphism given by $\Theta(f)(c\ot c') = \Theta(f)(c)(c')$. {\em Left contramodules} are defined by similar diagrams, in which $\Theta$ is replaced by the isomorphism $\Theta'(f)(c\ot c') = f(c')(c)$ (or equivalenty, as right contramodules for the co-opposite coalgebra $C^{op}$). Writing blanks for the arguments, and denoting by matching dots the respective functions $\alpha$ and their arguments, the definition of a right $C$-contramodule can be explicitly written as, for all $f\in \hom {C\ot C} M$, $m\in M$,
$$
\dot{\alpha}\left( \ddot{\alpha} \left (f\left (\dot{-}\ot \ddot{-}\right) \right)\right) = \alpha\left( f\left ((-)\sw 1\ot (-)\sw 2\right)\right), \qquad \alpha \left( \eC(-) m\right) = m.
$$
With the same conventions  the conditions for left contramodules are
$$
\dot{\alpha}\left( \ddot{\alpha} \left (f\left (\ddot{-}\ot \dot{-}\right) \right)\right) = \alpha\left( f\left ((-)\sw 1\ot (-)\sw 2\right)\right), \qquad \alpha \left( \eC(-) m\right) = m.
$$
If $N$ is a left $C$-comodule with coaction $\Nro: N\to C\ot N$, then its dual vector space $M = N^* := \hom N k$ is a right $C$-contramodule with the structure map
$$
\alpha : \hom C M \simeq \hom {C\ot N} k \to \hom N k = M, \qquad \alpha = \hom \Nro k.
$$
Explicitly, $\alpha$ sends a functional $f$ on $C\ot N$ to the functional $\alpha(f)$ on $N$, 
$$\alpha(f)(x) = f(x\sw{-1}\ot x\sw 0), \qquad x\in N.
$$
The dual vector space of a right $C$-comodule $N$ with a coaction $\roN: N\to N\ot C$ is a left $C$-contramodule with the structure map $\alpha = \hom \roN k$. The reader interested in more detailed accounts of the contramodule theory is referred to \cite{BohBrz:mon}, \cite{Pos:hom}.

\end{statement}

\begin{statement} {\bf Anti-Yetter-Drinfeld contramodules.}   Given a Hopf algebra $H$ with a bijective antipode $S$, anti-Yetter-Drinfeld contramodules are defined as $H$-modules and $H$-contramodules with a compatibility condition. Very much the same as in the case of anti-Yetter-Drinfeld modules \cite{HajKha:ant} they come in four different flavours.

\begin{zlist}
\item A {\em left-left anti-Yetter-Drinfeld contramodule} is a left $H$-module (with the action denoted by a dot) and a left $H$-contramodule with the structure map $\alpha$, such that, for all $h\in H$ and $f\in \hom H M$,
$$
h\!\cdot\! \alpha(f) = \alpha\left( h\sw 2\!\cdot\! f\left( S^{-1}(h\sw 1) (-) h\sw 3\right)\right).
$$ 
$M$ is said to be {\em stable}, provided that, for all $m\in M$, $\alpha(r_m) = m$, where $r_m: H\to M$, $h\mapsto h\!\cdot\! m$.

\item A {\em left-right anti-Yetter-Drinfeld contramodule} is a left $H$-module  and a right $H$-contramodule, such that, for all $h\in H$ and $f\in \hom H M$,
$$
h\!\cdot\! \alpha(f) = \alpha\left( h\sw 2\!\cdot\! f\left( S(h\sw 3) (-) h\sw 1\right)\right).
$$ 
$M$ is said to be {\em stable}, provided that, for all $m\in M$, $\alpha(r_m) = m$.

\item A {\em right-left anti-Yetter-Drinfeld contramodule} is a right $H$-module  and a left $H$-contramodule, such that, for all $h\in H$ and $f\in \hom H M$,
$$
\alpha(f)\!\cdot\! h = \alpha\left( f\left(h\sw 3 (-) S(h\sw 1)\right) \!\cdot\! h\sw 2 \right).
$$ 
$M$ is said to be {\em stable}, provided that, for all $m\in M$, $\alpha(\ell_m) = m$, where $\ell_m : H\to M$, $h\mapsto m\!\cdot\! h$. 

\item A {\em right-right anti-Yetter-Drinfeld contramodule} is a right $H$-module and a right $H$-contramodule, such that, for all $h\in H$ and $f\in \hom H M$,
$$
\alpha(f)\!\cdot\! h = \alpha\left( f\left(h\sw 1 (-) S^{-1}(h\sw 3)\right) \!\cdot\! h\sw 2 \right).
$$ 
$M$ is said to be {\em stable}, provided that, for all $m\in M$, $\alpha(\ell_m) = m$.
\end{zlist} 

In a less direct, but more formal way, the compatibility condition for left-left anti-Yetter-Drinfeld contramodules can be stated as follows. For all $h\in H$ and $f\in \hom H M$, define $k$-linear maps $\ell\ell_{f,h}: H\to M$, by 
$$
\ell\ell_{f,h}: h'\mapsto h\sw 2\!\cdot \!f\left( S^{-1}(h\sw 1) h' h\sw 3\right).
$$
 Then the main condition in (1) is 
 $$
 h\!\cdot\! \alpha(f) = \alpha\left(\ell\ell_{f,h}\right), \qquad \forall h\in H, \ f\in \hom H M .
 $$
Compatibility conditions between action and the structure maps $\alpha$ in (2)--(4) can be written in analogous ways. 

If $N$ is an anti-Yetter-Drinfeld module, then its dual $M=N^*$ is an anti-Yetter-Drinfeld contramodule (with the sides interchanged). Stable anti-Yetter-Drinfeld modules correspond to stable contramodules. For example, consider a right-left Yetter-Drinfeld module $N$. The compatibility between the right action and left coaction $\Nro$ thus is, for all $x\in N$ and $h\in H$,
$$
\Nro (x\!\cdot\! h) = S(h\sw 3) x\sw {-1}h\sw 1 \ot x\sw 0 h\sw 2.
$$
The dual vector space $M= N^*$ is a left $H$-module by $h\ot m \mapsto h\cdot m$, 
$$
(h\!\cdot\! m) (x) =m(x\!\cdot\! h),
$$
for all $h\in H$, $m\in M =\hom N k$ and $x\in N$, and a right $H$-contramodule with the structure map $\alpha(f) = f\circ \Nro$, $f\in \hom {H\ot N} k \simeq \hom H M$. The space $\hom {H\ot N} k$ is a left $H$-module by $(h\!\cdot\! f)(h'\ot x) = f(h'\ot x\!\cdot\! h)$. Hence
$$
(h\!\cdot\! \alpha(f) )(x) = \alpha(f)(x\!\cdot\! h) = f\left(\Nro \left( x\!\cdot\! h\right) \right),
$$
and
\begin{eqnarray*}
\alpha\left( h\sw 2\!\cdot\! f\left( S(h\sw 3) (-) h\sw 1\right)\right)(x) &=&  h\sw 2\!\cdot\! f\left( S(h\sw 3) x\sw{-1} h\sw 1 \ot x\sw 0\right) \\
&=&  f\left( S(h\sw 3) x\sw{-1} h\sw 1 \ot x\sw 0\!\cdot\! h\sw 2\right).
\end{eqnarray*}
Therefore, the compatibility condition in item (2) is satisfied. The $k$-linear map $r_m : H\to M$ is identified with $r_m: H\ot N \to k$, $r_m(h\ot x) = m(x\!\cdot\! h)$. In view of this identification, the stability condition comes out as, for all $m\in M$ and $x\in N$,
$$
m(x) = \alpha(r_m)(x) = r_m(x\sw{-1}\ot x\sw 0) =  m(x\sw 0\!\cdot\! x\sw{-1}),
$$
and is satisfied provided $N$ is a stable right-left anti-Yetter-Drinfeld module. Similar calculations establish connections between other versions of anti-Yetter-Drinfeld modules and contramodules.
\end{statement}
\begin{statement} {\bf Hopf-cyclic homology of module coalgebras.}  
Let $C$ be a left $H$-module coalgebra. This means that $C$ is a coalgebra and a left $H$-comodule such that, for all $c\in C$ and $h\in H$,
$$
\DC (h\!\cdot\! c) = h\sw 1\!\cdot\! c\sw 1\ot h\sw 2\!\cdot\! c\sw 2, \qquad \eC(h\!\cdot\! c) = \eH(h)\eC(c).
$$
The multiple tensor product of $C$, $C^{\otimes n+1}$, is  a left $H$-module by the diagonal action, that is
$$
h\!\cdot\! (c^0\ot c^ 1\ot \ldots \ot c^n):= h\sw 1\!\cdot\! c^0\ot h\sw 2\!\cdot\! c^ 1\ot \ldots \ot h\sw{n+1}\!\cdot\! c^n.
$$
Let $M$ be a stable left-right anti-Yetter-Drinfeld contramodule. For all positive integers $n$, set $C_n^H(C,M) := \rhom H {C^{\otimes n+1}} M$ (left $H$-module maps),  and, for all $0\leq i,j\leq n$, define
$
d_i: C_n^H(C,M)\to C_{n-1}^H(C,M)$, $s_j: C_n^H(C,M)\to C_{n+1}^H(C,M)$, $t_n: C_n^H(C,M)\to C_{n}^H(C,M)$, 
by
\begin{eqnarray*}
d_i(f)(c^0,\ldots , c^{n-1}) &=& f(c^0, \ldots, \DC(c^i),\ldots , c^{n-1}), \qquad 0\leq i < n,\\
d_n(f)(c^0,\ldots , c^{n-1}) &=& \alpha\left(f\left( c^0\sw 2, c^1, \ldots , c^{n-1}, (-)\!\cdot\! c^0\sw 1\right)\right),\\
s_j (f) (c^0,\ldots , c^{n+1}) &=& \eC(c^{j+1})f(c^0,\ldots ,c^j, c^{j+2}, \ldots ,c^{n+1}),\\
t_n (f)(c^0,\ldots , c^{n}) &=&\alpha\left(f\left( c^1, \ldots , c^{n}, (-)\!\cdot\! c^0\right)\right) .
\end{eqnarray*}
It is clear that all the maps $s_j$, $d_i$, $i<n$, are well-defined, i.e.\ they send left $H$-linear maps to left $H$-linear maps. That $d_n$ and $t_n$ are well-defined follows by the anti-Yetter-Drinfeld condition. To illustrate how the anti-Yetter-Drinfeld condition enters here we check that the $t_n$ are well defined. For all $h\in H$,
\begin{eqnarray*}
t_n (f)%&&\hspace{-1cm}
(h\!\cdot\! (c^0,\ldots , c^{n})) &=& t_n (f)(h\sw 1\!\cdot\! c^0,\ldots , h\sw{n+1}\!\cdot\! c^{n})\\
&=&\!\!\! \alpha\left(f\left( h\sw{2}\!\cdot\! c^1, \ldots ,  h\sw{n+1}\!\cdot\! c^n,  (-)h\sw 1\!\cdot\! c^0\right)\right)\\
&=&\!\!\! \alpha\left(f\left(  h\sw{2}\!\cdot\! c^1, \ldots , h\sw{n+1}\!\cdot\! c^n,   h\sw{n+2} S(h\sw{n+3})(-)h\sw 1\!\cdot\! c^0\right)\right)
\\
&=&\!\!\! \alpha\left(h\sw 2\!\cdot\! f\left(c^1, \ldots , c^{n},    S(h\sw{3})(-)h\sw 1\!\cdot\! c^0\right)\right)\\
&=&\!\!\! h\!\cdot\! \alpha\left(f\left( c^1, \ldots , c^{n}, (-)\!\cdot\! c^0\right)\right)
= h\!\cdot\! t_n (f)(c^0,\ldots , c^{n}),
\end{eqnarray*}
where the third equation follows by the properties of the antipode and counit, the fourth one is a consequence of the $H$-linearity of $f$, while the anti-Yetter-Drinfeld condition is used to derive the penultimate equality. 

\begin{theorem}
Given a left $H$-module coalgebra $C$ and a left-right stable anti-Yetter-Drinfeld contramodule $M$, $C^H_*(C,M)$ with the $d_i$, $s_j$, $t_n$ defined above is a cyclic module.
\end{theorem}

\begin{proof} One needs to check whether the maps $d_i$, $s_j$, $t_n$ satisfy the relations of a cyclic module; 
see e.g.\ \cite[p.\ 203]{Lod:cyc}. Most of the calculations are standard, we only display examples of those which make use of the contramodule axioms. For example,
\begin{eqnarray*}
(t_{n-1}\circ d_{n-1})(f)%&&\hspace{-1cm}
(c^0,\ldots , c^{n-1}) &=& \alpha\left(d_{n-1}(f)\left( c^1, \ldots , c^{n-1}, (-)\!\cdot\! c^0\right)\right)\\
&=& \alpha\left(f\left( c^1, \ldots , c^{n-1}, \DC\left((-)\!\cdot\! c^0\right)\right)\right)\\
&=& \alpha\left(f\left( c^1, \ldots , c^{n-1}, (-)\sw 1\!\cdot\! c^0\sw 1,(-)\sw 2\!\cdot\! c^0\sw 2\right)\right)\\
&=& \dot{\alpha}\left( \ddot{\alpha}\left(f\left( c^1, \ldots , c^{n-1}, \dot{(-)}\!\cdot\! c^0\sw 1,\ddot{(-)}\!\cdot\! c^0\sw 2\right)\right)\right)\\
&=& \alpha\left(t_n(f)\left( c^{0}\sw 2, c^1, \ldots , c^{n-1}, (-)\!\cdot\! c^0\sw 1\right)\right)\\
&=& (d_{n}\circ t_{n})(f)(c^0,\ldots , c^{n-1}), 
\end{eqnarray*}
where the third equality follows by the module coalgebra property of $C$, and the fourth one is a consequence of the associative law for contramodules. In a similar way, using compatibility of $H$-action on $C$ with counits of $H$ and $C$, and that $\alpha \left( \eC(-) m\right) = m$, for all $m\in M$, one easily shows that $d_{n+1}\circ s_n$ is the identity map on  $C^H_n(C,M)$. The stability of $M$ is used to prove that $t_n^{n+1}$ is the identity. Explicitly,
\begin{eqnarray*}
t_n^{n+1} (f)&&\hspace{-1cm}(c^0,\ldots , c^{n}) = \alpha^{n+1}(f((-)\!\cdot\! c^0,\ldots , (-)\!\cdot\! c^{n}))\\
&=&\alpha(f((-)\sw 1\!\cdot\! c^0,\ldots , (-)\sw{n+1}\!\cdot\! c^{n}))
= \alpha(r_{f(c^0,\ldots , c^{n})}) = f(c^0,\ldots , c^{n}),
\end{eqnarray*}
where the second equality follows by the $n$-fold application of the associative law for contramodules, and the penultimate equality is a consequence of the $H$-linearity of $f$. The final equality follows by the stability of $M$.
\end{proof}

Let $N$ be a right-left stable anti-Yetter-Drinfeld module, and $M=N^*$ be the corresponding left-right stable anti-Yetter-Drinfeld contramodule, then
$$
C^H_n(C,M) =  \rhom H {C^{\otimes n+1}} {\hom Nk} \simeq \hom {N\ot_HC^{\otimes n+1}}k.
$$
With this identification, the cycle module $C^H_n(C,N^*)$ is obtained  by applying functor $\hom - k$ to the cyclic module for $N$ described in \cite[Theorem~2.1]{HajKha:Hop}.
\end{statement}

\begin{statement} {\bf Hopf-cyclic cohomology of module algebras.}   Let $A$ be a left $H$-module algebra. This means that $A$ is an algebra and a left $H$-module such that, for all $h\in H$ and $a,a'\in A$,
$$
h\!\cdot\! (aa') = (h\sw 1\!\cdot\! a)(h\sw 2\!\cdot\! a), \qquad h\!\cdot\! 1_A = \eH(h)1_A.
$$

\begin{lemma}\label{lemma}  Given a left $H$-module algebra $A$ and a left $H$-contramodule, $\hom A M$ is an $A$-bimodule with the left and right $A$-actions defined by
$$
(a\!\cdot\! f)(b) = f(ba), \qquad (f\!\cdot\! a)(b) = \alpha\left( f\left( \left( (-)\!\cdot\! a\right)b\right)\right),
$$
for all $a,b\in A$ and $f\in \hom A M$.
\end{lemma}

\begin{proof} The definition of left $A$-action is standard, compatibility between left and right actions is immediate. To prove the associativity of the right $A$-action, take any $a,a',b\in A$ and $f\in \hom A M$, and compute
\begin{eqnarray*}
\left(\left(f\!\cdot\! a\right)\!\cdot\! a'\right)(b) &=& \dot{\alpha}\left( \ddot{\alpha} \left( f\left( \left( \ddot{(-)}\!\cdot\! a\right)\left(\dot{(-)}\!\cdot\! a'\right) b\right)
\right)\right)\\
&=& \alpha \left( f\left(\left( (-)\sw 1\!\cdot\! a\right)\left((-)\sw 2\!\cdot\! a'\right)b\right)\right)\\
&=& \alpha \left( f\left(\left( (-)\!\cdot\! \left( aa'\right)\right) b\right)\right) = \left(\left(aa'\right)\!\cdot\! f\right)(b),
\end{eqnarray*}
where the second equality follows by the definition of a left $H$-contramodule, and the third one in a consequence of the module algebra property. The unitality of the right $A$-action follows by the triangle diagram for contramodules and the fact that $h\!\cdot\! 1_A = \eH(h)1_A$.
\end{proof}

For an $H$-module algebra $A$, $A^{\otimes n+1}$ is a left $H$-module by the diagonal action
$$
h\!\cdot\! (a^0\ot a^1\ot \ldots \ot a^n):= h\sw 1\!\cdot\! a^0\ot h\sw 2\!\cdot\! a^ 1\ot \ldots \ot h\sw{n+1}\!\cdot\! a^n.
$$
Take a stable left-left  anti-Yetter-Drinfeld contramodule $M$, set $C^n_H(A,M)$ to be the space of left $H$-linear maps $\rhom H {A^{\otimes n+1}} M$,  and, for all $0\leq i,j\leq n$, define
$
\delta_i: C^{n-1}_H(A,M)\to C^{n}_H(A,M)$, $\sigma_j: C^{n+1}_H(A,M)\to C^{n}_H(C,M)$, $\tau_n: C^n_H(A,M)\to C^n_H(A,M)$, 
by
\begin{eqnarray*}
\delta_i(f)(a^0,\ldots , a^{n}) &=& f(a^0, \ldots, a^{i-1}, a^ia^{i+1}, a^{i+2},\ldots , a^{n}), \qquad 0\leq i < n, \\
\delta_n(f)(a^0,\ldots , a^{n}) &=& \alpha\left(f\left(\left((-)\!\cdot\! a^n\right) a^0, a^1, \ldots , a^{n-1}\right)\right),\\
\sigma_j (f) (c^0,\ldots , c^{n}) &=& f(a^0,\ldots ,a^j, 1_A, a^{j+1}, \ldots ,a^{n}),\\
\tau_n(f)(a^0,\ldots , a^{n}) &=& \alpha\left(f\left((-)\!\cdot\! a^n, a^0, a^1, \ldots , a^{n-1}\right)\right) .
\end{eqnarray*}
Similarly to the module coalgebra case, the above maps are well-defined by the anti-Yetter-Drinfeld condition. Explicitly, using the aformentioned condition as well as the fact that the inverse of the antipode is the antipode for the co-opposite Hopf algebra, one  computes
\begin{eqnarray*}
\tau_n(f)&&\hspace{-1cm}(h\!\cdot\!(a^0,\ldots , a^{n}) ) = \alpha\left(f\left(((-)h\sw{n+1})\!\cdot\! a^n, h\sw{1}\!\cdot\! a^0, h\sw{2}\!\cdot\! a^1, \ldots , h\sw{n}\!\cdot\! a^{n-1}\right)\right)\\
&=& \alpha\left(f\left((h\sw 2 S^{-1}(h\sw 1)(-)h\sw{n+3})\!\cdot\! a^n, h\sw{1}\!\cdot\! a^0, h\sw{2}\!\cdot\! a^1, \ldots , h\sw{n+2}\!\cdot\! a^{n-1}\right)\right)\\
&=&\alpha\left(h\sw 2\!\cdot\! f\left((S^{-1}(h\sw 1)(-)h\sw{3})\!\cdot\! a^n,  a^0, \ldots ,  a^{n-1}\right)\right)
= h\!\cdot\!\tau_n(f)(a^0,\ldots , a^{n} ).
\end{eqnarray*} 
Analogous calculations ensure that also $\delta_n$ is well-defined.

\begin{theorem}\label{thm.alg}  Given a left $H$-module algebra $A$ and a stable left-left  anti-Yetter-Drinfeld contramodule $M$, $C_H^*(A,M)$ with the $\delta_i$, $\sigma_j$, $\tau_n$ defined above is a (co)cyclic module.
\end{theorem}

\begin{proof}
In view of Lemma~\ref{lemma} and  taking into account the canonical  isomorphism $\hom {A^{\otimes n+1}} M \simeq \hom {A^{\otimes n}} {\hom A M}$,
$$
\hom {A^{\otimes n+1}} M \ni f\mapsto \left[a^1\ot a^2\ot \ldots \ot a^n \mapsto f\left(-,a^1,a^2,\ldots, a^n\right)\right],
$$
 the simplicial part comes from the standard $A$-bimodule cohomology. Thus only the relations involving $\tau_n$  need to be checked. In fact only the equalities $\tau_n \circ \delta_n = \delta_{n-1}\circ \tau_{n-1}$ and $\tau_n^{n+1} =\id$  require one to make use of definitions of a module algebra and a left contramodule.  In the first case, for all $f\in C_H^n(A,M)$,
 \begin{eqnarray*}
 (\tau_n \circ\delta _n )(f)(a^0,\ldots ,a^n) &=& \dot{\alpha}\left(\ddot{\alpha}\left( f\left(\left( \ddot{(-)}\!\cdot \! a^{n-1}\right)\left(\dot{(-)}\!\cdot\! a^n\right),a^0,\dots ,a^{n-2}\right)\right)\right)\\
 &=& \alpha\left( f\left(\left((-)\sw 1\!\cdot \! a^{n-1}\right)\left((-)\sw 2\!\cdot\! a^n\right),a^0,\dots ,a^{n-2}\right)\right)\\
 &=&  \alpha\left( f\left((-)\!\cdot \! \left(a^{n-1}a^n\right),a^0,\dots ,a^{n-2}\right)\right)\\
 &=& (\delta_{n-1} \circ\tau _{n-1} )(f)(a^0,\ldots ,a^n), 
 \end{eqnarray*}
 where the second equality follows by the associative law for left contramodules and the third one by the definition of a left $H$-module algebra. The equality $\tau_n^{n+1} =\id$  follows by the associative law of contramodules, the definition of left $H$-action on $A^{\otimes n+1}$, and by the stability of anti-Yetter-Drinfeld contramodules.
\end{proof}

In the case of a contramodule $M$ constructed on the dual vector space of a stable right-right anti-Yetter-Drinfeld module $N$, the complex described in Theorem~\ref{thm.alg} is the right-right version of Hopf-cyclic complex of a left module algebra with coefficients in $N$ discussed in \cite[Theorem~2.2]{HajKha:Hop}.

\end{statement}
\begin{statement} {\bf Anti-Yetter-Drinfeld contramodules and hom-connections.}
  Anti-Yetter-Drinfeld modules over a Hopf algebra $H$ can be understood as comodules of an $H$-coring; see \cite{Brz:fla} for explicit formulae and \cite{BrzWis:cor} for more information about corings. These are corings with a group-like element, and thus their comodules can be interpreted as modules with a flat connection; see \cite{Brz:fla} for a review. Consequently, anti-Yetter-Drinfeld modules are modules with a flat connection (with respect to a suitable differential structure); see \cite{KayKha:Hop}. 

Following similar line of argument anti-Yetter-Drinfeld contramodules over a Hopf algebra $H$ can be understood as contramodules of an $H$-coring. This is a coring of an entwining type, as a vector space built on $H\ot H$, and its form is determined by the anti-Yetter-Drinfeld compatibility conditions between action and contra-action. The coring $H\ot H$ has a group-like element $1_H\ot 1_H$, which induces a differential graded algebra structure on tensor powers of the kernel of the counit of $H\ot H$. As explained in \cite[Section~3.9]{Brz:non} contramodules of a coring with a group-like element correspond to  {\em flat hom-connections}. Thus, in particular, anti-Yetter-Drinfeld contramodules are  flat hom-connections. We illustrate this discussion by the example of right-right anti-Yetter-Drinfeld contramodules. 

First recall the definition of hom-connections from \cite{Brz:non}. Fix a differential graded algebra $\Omega A$ over an algebra $A$. A {\em hom-connection}  is a pair $(M,\nabla_0)$, where $M$ is a right $A$-module and $\nabla_0$ is a $k$-linear map from the space of right $A$-module homomorphisms $\rhom A {\Omega^1 A} M$ to $M$, $\nabla_0: \rhom A {\Omega^1 A} M \to M$, such that, for all $a\in A$, $f\in \rhom A {\Omega^1 A} M $, 
$$
\nabla_0(f\!\cdot\! a) = \nabla_0(f)\!\cdot\! a +f(da),
$$
where $f\!\cdot \!a\in \rhom A {\Omega^1 A}  M$ is given  by $f\!\cdot \!a: \omega \mapsto f(a\omega)$, and $d: \Omega^*A\to \Omega^{*+1}A$ is the differential.  Define $\nabla_1 : \rhom A {\Omega^2 A} M \to\rhom A {\Omega^1 A} M$, by
$
\nabla_1(f)(\omega) = \nabla_0(f\!\cdot\!\omega) + f(d\omega),
$
where, for all $f\in \rhom A {\Omega^2 A} M$, the map $f\!\cdot\!\omega \in \rhom A {\Omega^1 A} M$ is given by $\omega'\mapsto f(\omega\omega')$. The composite $F= \nabla_0\circ\nabla_1$ is called the {\em curvature} of $(M,\nabla_0)$. The hom-connection $(M,\nabla_0)$ is said to be {\em flat} provided its curvature is equal to zero.

Consider a Hopf algebra $H$ with a bijective antipode, and define an $H$-coring $\cC= H\ot H$ as follows. The $H$ bimodule structure of  $\cC $ is given by
$$
h\!\cdot\! (h'\ot h'') = h\sw 1 h' S^{-1}(h\sw 3)\ot h\sw 2 h'', \qquad (h'\ot h'')\!\cdot\! h = h'\ot h''h,
$$
the coproduct is $\DH\ot \id_H$ and counit $\eH\ot \id_H$. Take a right $H$-module $M$. The identification of right $H$-linear maps $H\ot H \to M$ with $\hom H M$ allows one to identify right contramodules of the $H$-coring $\cC$ with right-right anti-Yetter-Drinfeld contramodules over $H$.

 The kernel of the counit in $\cC$ coincides with $H^+\ot H$, where $H^+ = \ker \eH$. Thus the associated differential graded algebra over $H$ is given by $\Omega^nH = (H^+\ot H)^{\ot_H n} \simeq  (H^+)^{\ot n}\ot H$, with the differential given on elements $h$ of $H$ and one-forms $h'\ot h \in H^+\ot H$ by
$$
dh = 1_H\ot h - h\sw 1S^{-1}(h\sw 3)\ot h\sw 2, 
$$
$$
 d(h'\ot h) = 1_H\ot h'\ot h - h'\sw 1\ot h'\sw 2\ot h + h'\ot h\sw 1S^{-1}(h\sw 3)\ot h\sw 2.
$$
Take a right-right anti-Yetter-Drinfeld contramodule $M$ over a Hopf algebra $H$ and identify $\rhom H {\Omega^1 H} M $ with $\hom {H^+} M$. For any $f\in \hom {H^+} M$, set $\bar{f}: H\to M$ by $\bar{f}(h) = f(h-\eH(h)1_H)$, and then define
$$
\nabla_0: \hom {H^+} M\to M, \qquad \nabla_0(f) = \alpha (\bar{f}).
$$
$(M,\nabla_0)$ is a flat hom-connection with respect to the differential graded algebra $\Omega H$.
\end{statement}
\begin{statement} {\bf Final remarks.}  
In this note  a new class of coefficients for the Hopf-cyclic homology was introduced. 
It is an open question to what extent Hopf-cyclic homology with coefficients in anti-Yetter-Drinfeld contramodules is useful in studying problems arising in (non-commutative) geometry. The answer is likely to depend on the supply of (calculable) examples, such as those coming from the transverse index theory of foliations (which motivated the introduction of Hopf-cyclic homology in \cite{ConMos:cyc}). It is also likely to depend on the structure of Hopf-cyclic homology with contramodule coefficients. One can easily envisage that, in parallel to the theory with anti-Yetter-Drinfeld module coefficients, the cyclic theory described in this note admits cup products (in the case of module coefficients these were foreseen in \cite{HajKha:Hop} and constructed in \cite{KhaRan:cup}) or homotopy  formulae of the type discovered for anti-Yetter-Drinfeld modules in \cite{MosRan:cyc}. Alas, these topics go beyond the scope of this short note. The author is convinced, however, of the worth-whileness of investigating them further. 
\end{statement}

 \end{document}